\documentclass[12pt,a4paper]{amsart}
\usepackage{amsmath,amssymb,amscd,amsfonts}
\usepackage{layout,hyperref}
\usepackage[textwidth=3.7cm]{todonotes}
\usepackage{anysize}
\usepackage{amsmath}
\usepackage{graphicx, amssymb, color}

\def\Im{\operatorname{Im}}
\def\Re{\operatorname{Re}}

\newcommand{\Z}{\mathbb{Z}}
\newcommand{\C}{\mathbb{C}}
\newcommand{\N}{\mathbb{N}}
\newcommand{\QQ}{\mathbb{Q}}
\newcommand{\HH}{\mathbb{H}}

\newcommand{\R}{\mathbb{R}}
\newcommand{\E}{\mathcal{E}}

\newcommand{\conj}{\overline}
\newcommand{\bfo}{\mathbf{1}}
\newcommand{\ord}{\mathrm{ord}}
\theoremstyle{plain}
\newtheorem{theorem}{Theorem}[section]

\newtheorem{proposition}{Proposition}[section]

\theoremstyle{remark}
\newtheorem*{remark}{Remark}
\theoremstyle{definition}
\theoremstyle{definition}

\author[Herrero]{Sebasti\'an Herrero}
\address{
Sebasti\'an Herrero,
Department of Mathematical Sciences, 
Chalmers University of Technology and University of Gothenburg,
SE-412 96 Gothenburg, 
Sweden.}
\email{sebastian.herrero.m@gmail.com}

\author[von Pippich]{Anna-Maria von Pippich}
\address{
Anna-Maria von Pippich,
Fachbereich Mathematik, 
Technische Universit\"at Darmstadt,
Schlo{\ss}gartenstra{\ss}e 7, 
D-64289 Darmstadt, 
Germany.
}
\email{pippich@mathematik.tu-darmstadt.de}

\begin{document}

\title[Mock modular forms whose shadows are Eisenstein series]{Mock modular forms whose shadows are Eisenstein series of integral weight} 





\maketitle

\begin{abstract}
The purpose of this article is to give a simple and explicit construction of mock modular forms whose shadows are Eisenstein series of arbitrary integral weight, level, and character. As application, we construct forms whose shadows are Hecke's Eisenstein series of weight one associated to imaginary quadratic fields, recovering some results by Kudla, Rapoport and Yang (1999), and Schofer (2009), and forms whose shadows equal $\Theta^{2k}(z)$ for $k\in \{1,2,3,4\}$, where 
$\Theta(z)$ denotes Jacobi's theta function.
\end{abstract}


\section{Introduction}

\subsection{Holomorphic, harmonic and mock modular forms} 
Throughout this article, let $k$, $N$ be integers with $N\geq 1$ and let $\chi$ denote a Dirichlet character mod $N$ satisfying $\chi(-1)=(-1)^k$.
Further, let $\HH=\{z=x+iy\in \C: x,y\in\R, y>0\}$ 
denote the complex upper-half plane and let $$\Gamma_0(N):=\left\{\pm \begin{pmatrix}a & b \\ c& d\end{pmatrix}\in \mathrm{PSL}_2(\Z):c\equiv 0 \, (\mbox{mod } N)\right\}.$$
 A modular form (resp.~cusp form) of weight $k$, level $N$ and character $\chi$ is a holomorphic function $f:\HH\to \C$ satisfying the transformation property
\begin{equation}\label{eq-modularity}
f\left(\gamma z\right)=\chi(d)(c\tau+d)^kf(z) \, \mbox{ for all }\gamma=\pm \bigl(\begin{smallmatrix}a & b \\ c& d\end{smallmatrix}\bigr)\in \Gamma_0(N),
\end{equation} 
where $\gamma z:=(az+b)(cz +d)^{-1}$, which is also holomorphic (resp.~vanishes) at every cusp of $\Gamma_0(N)$. We denote by $M_k(N,\chi)$ the finite dimensional $\C$-vector space of holomorphic modular forms of weight $k$, level $N$ and character $\chi$, and by $S_k(N,\chi)$ the subspace of cusp forms. Similarly, a weakly harmonic Maass form of weight $k$, level $N$ and character $\chi$ is a real-analytic function  $f:\HH\to \C$ satisfying \eqref{eq-modularity}, with
$\Delta_k(f)=0$, where 
$$\Delta_k:=y^2\left(\frac{\partial^2}{\partial x^2}+\frac{\partial^2}{\partial y^2}\right)-iky\left(\frac{\partial}{\partial x}+i\frac{\partial}{\partial y}\right)$$
is the weight $k$ hyperbolic Laplacian, and with at most linear exponential growth at the cusps of $\Gamma_0(N)$. We denote by $H_{k}(N,\chi)$ the space of weakly harmonic Maass forms of weight $k$, level $N$ and character $\chi$, whose images under the differential operator
$$\xi_{k}:=2iy^{k}\overline{\partial_{\overline{z}}}$$
are holomorphic at every cusp of $\Gamma_0(N)$. By introducing certain multiplier system one can define holomorphic modular forms and harmonic Maass forms of half integral weight (see, e.g., \cite{BFOR} Chapter 4), but these will play a minor role in this paper.\\
The differential operator $\xi_{k}$ is related to the Laplacian $\Delta_k$ by the formula
\begin{equation}\label{eq-Delta-xi}
\Delta_k=\xi_{2-k}\circ \xi_k.
\end{equation}
It is known that every function $F(z)\in H_{2-k}(N,\chi)$ has a Fourier expansion at the cusp $\infty$ of the form
\begin{equation}\label{Four-exp-harmonic}
F(z)=\sum_{n= n_0}^{\infty}c^+(n)q^n-\sum_{n=0}^{\infty}\overline{c(n)}\beta_{2-k}(n,y)q^{-n},
\end{equation}
where $n_0\in \Z$ and $q:=e^{2\pi i z}$. Here, for $n>0$, we have
\begin{align}\label{def_beta1}
\beta_{2-k}(n,y):=\int_y^{\infty}e^{-4\pi nu}u^{k-2}du,
\end{align}
and, for $n=0$, 
\begin{align}\label{def_beta2}
\beta_{2-k}(0,y):=\left\{ 
\begin{array}{ll}
(1-k)^{-1}y^{k-1}, & \mbox{ if } k\neq 1,\\
-\log(y),&  \mbox{ if } k=1.
\end{array}\right.
\end{align}
The coefficients $c^+(n)\in\C$ (resp.~$c(n)\in\C$) are called holomorphic (resp.~non-holomorphic) coefficients of $F(z)$. 
One has
$$\xi_{2-k}(F)(z)=\sum_{n=0}^{\infty}c(n)q^n \in M_k(N,\overline{\chi}),$$
and one says that the function
$$\sum_{n= n_0}^{\infty}c^+(n)q^n$$
is a mock modular form of weight $2-k$, level $N$ and character $\chi$ with shadow $\xi_{2-k}(F)$. Note that a mock modular form is determined by its shadow up to addition of a weakly holomorphic modular form of weight $2-k$ (this is a meromorphic modular form whose poles are supported at cusps).
 
In general, it is a difficult problem to describe the Fourier coefficients of mock modular forms. In the case of half integral weight, we have the influential work of Zwegers \cite{Z} that involves the mock-theta functions of Ramanujan, and also  the work of Bringmann and Ono \cite{BO}, Bruinier and Ono \cite{BruOn}, and Bruinier \cite{Bru}, among others. An explicit construction of a mock modular form whose shadow equals $\Theta^3(z)$, where $\Theta(z)$ denotes Jacobi's theta function, has been constructed by Rhoades and Waldherr in \cite{RW}. 

In the case of integral weight $2-k$ with $k\geq 2$, we have the work of Bruinier, Ono and Rhoades \cite{BOR} relating the non-constant Fourier coefficients of mock modular forms of weight $2-k$ to the coefficients of weakly holomorphic modular forms of weight $k$. More precisely, their result states that for $F\in H_{2-k}(N,\chi)$ with Fourier expansion $\eqref{Four-exp-harmonic}$, one has
$$D^{k-1}(F)=\sum_{n=n_0}^{\infty}c^+(n)n^{k-1}q^n\in M_k^{!}(N,\chi),$$
where $D:=\frac{1}{2\pi i }\partial_z$ and $M_k^{!}(N,\chi)$ denotes the space of weakly holomorphic modular forms of weight $k$, level $N$ and character $\chi$ (see,  \begin{it}loc.~cit.\end{it}, Theorem 1.1, p.~675). This relation  determines the coefficients $c^+(n)$ for $n\neq 0$ provided one has an explicit formula for the Fourier coefficients of $D^{k-1}(F)$. Unfortunately, in general, this seems to be a difficult problem (see, e.g., \begin{it}loc.~cit.\end{it}, pp.~675-676, where some examples are discussed).

In the case of weight one (the \begin{it}self dual \end{it}case), there is the work of  Duke and Li \cite{DL15}, where the authors investigate certain mock modular forms associated to imaginary quadratic fields, whose shadows are cuspidal new forms, and relate their Fourier coefficients to logarithms of  absolute values of units in the ring of integers of the corresponding Hilbert class field. We would also like to mention the recent work of Li \cite{YL17}, where he constructs a mock modular form  whose shadow is an Eisenstein series of weight one studied by Hecke, which is a holomorphic vector-valued modular form for the Weil representation.

In this paper, we focus exclusively on the simplest case concerning mock modular forms whose shadows are Eisenstein series of integral weight. One of the first examples of such a form is given by the series
$$
G_2(z):=\frac{\pi}{3}-8\pi\sum_{n=1}^{\infty}\sigma_1(n)q^n,
$$
where $\sigma_1(n):=\sum_{d\mid n}d$. The function $G_2(z)$ is a mock modular form of weight 2 with shadow equal to 1 (one can consider 1 as an Eisenstein series of weight zero and level one). Mock modular forms whose shadows are Eisenstein series of weight $k\geq 4$ and level one have been constructed by Bringmann, Guerzhoy, Kent and Ono  in \cite{BGKO} (see Theorem 2.1, p.~1100). Here, the authors needed to control the constant term of any given harmonic Maass form in order to compute a certain extended Petersson inner product between weakly holomorphic modular forms. These computations are also contained in \cite{BFOR} (see Theorem 6.15, p.~104).

\subsection{Purpose of the article} The purpose of this article is to give a simple and explicit construction of mock modular forms whose shadows are Eisenstein series of arbitrary integral weight $k\geq 1$, level $N\geq 1$ and character $\chi$. To be more precise, recall that, given $f\in M_k(N,\chi)$ and $g\in S_k(N,\chi)$, one can define the Petersson inner product
$$\langle f,g\rangle :=\frac{1}{[\Gamma_0(1):\Gamma_0(N)]}\int_{\Gamma_0(N)\backslash \HH}f(z)\overline{g(z)}y^{k-2} dxdy.$$
Then, a classical result of the theory of modular forms states that the subspace
$$N_k(N,\chi):=\{f\in M_k(N,\chi):\langle f,g\rangle=0 \mbox{ for all }g\in S_k(N,\chi)\}$$
is generated by certain Eisenstein series $E_k^{\psi,\varrho,t}(z)$, where $\psi$ and $\varrho$ are Dirichlet characters mod $L$ and $M$, respectively, satisfying $\psi \varrho=\chi$, and $t\geq 1$ is an integer with $tLM\mid N$.  We refer to Section \ref{section-hol-Eis} for the precise definition of $E_k^{\psi,\varrho,t}(z)$. Note that one can assume that $\chi(-1)=(-1)^k$, since otherwise $M_k(N,\chi)=\{0\}$.

In this paper we shall describe a simple and explicit construction of a weight $2-k$ mock modular form
\begin{equation}\label{eq-Four-exp-mock-G}
G_{2-k}^{\psi,\varrho,t}(z)=\sum_{n=0}^{\infty}c^+_{2-k}(n,t,\psi,\varrho)q^n,  
\end{equation}
whose shadow equals $E_k^{\psi,\varrho,t}(z)$. If $k\geq 2$, the construction of the corresponding harmonic Maass form is fairly elementary and it is obtained by evaluating certain classical non-holomorphic Eisenstein series (going back, at least, to Hecke) at special points. 
Only the case $k=1$ requires some minor extra manipulations, since one has to evaluate a derivative of these non-holomorphic Eisenstein series, following ideas of Kudla, Rapoport and Yang \cite{KRY99}. This is done in detail in Section \ref{sect-const-preimage}.  The harmonic Maass form constructed here has only polynomial growth at the cusps of $\Gamma_0(N)$. If $k\geq 2$, these properties determine the mock modular form $G_{2-k}^{\psi,\varrho,t}(z)$ uniquely. If $k=1$, this mock modular form is unique up to addition of a form in $M_1(N,\overline{\chi})$. 

In order to state our results more precisely, let us introduce some notation. By $\bfo_N$ we denote the trivial Dirichlet character mod $N$. Furthermore, given a Dirichlet character $\chi$ mod $N$, we write $\chi^0$ for the unique primitive Dirichlet character inducing $\chi$. We denote by  $m_{\chi}$ the conductor of $\chi$ and we define $\ell_{\chi}:=N/m_{\chi}$. We also denote by
$$W(\chi^0):=\sum_{n=1}^{m_{\chi}}\chi^0(n)e^{2\pi i n/m_{\chi}}$$
the Gauss sum of $\chi^0$, and by
$$L(s,\chi):=\sum_{n=1}^{\infty}\chi(n)n^{-s}$$ 
the Dirichlet $L$-function associated to $\chi$, which converges absolutely for $s\in \C$ with $\Re(s)>1$. We convey that $f(x)=0$ if $x\not \in \N$, for any given arithmetic function $f:\N\to \C$.

Throughout the article, we let $\psi$ and $\varrho$ denote Dirichlet characters mod $L$ and $M$, respectively, satisfying 
\begin{equation}\label{eq-psi-rho-k-compatible}
\psi(-1)\varrho(-1)=(-1)^k,
\end{equation}
and we let $t\geq 1$ be an integer.
We set
\begin{align}\label{def_sigma_psi_rho}
\sigma^{\psi,\varrho}_{k-1}(n):=\sum_{0<c\mid n}\psi\left(\frac{n}{c}\right)c^{k-1}\sum_{0<d\mid \mathrm{gcd}(\ell_{\varrho},c)}d\, \mu\left(\frac{\ell_{\varrho}}{d}\right)
\overline{\varrho^0}\left(\frac{\ell_{\varrho}}{d}\right)\varrho^0\left(\frac{c}{d}\right).
\end{align}

The following are the main results of this paper.
\begin{theorem}\label{teo-fourier-exp-mock-geq2}
Assume $k>2$, or $k=2$ and $(\psi,\varrho)\neq (\bfo_L,\bfo_M)$. Put
$$c_{2-k}^+(0,t,\psi,\varrho):=\left\{
\begin{array}{ll}
\dfrac{ 2^{2-k}\pi i^{k-2}}{(tM)^{k-1}(k-1)}L(k-1,\overline{\psi})\prod_{p\mid M}(1-p^{-1}), & \mbox{ if } \varrho=\bfo_M,\\
0, & \mbox{ if } \varrho\neq \bfo_M,
\end{array}
\right.$$
and for $n>0$ put
$$c_{2-k}^+(n,t,\psi,\varrho):=
\frac{ 2^{2-k} \pi i^{k-2}}{M^k(k-1)}W(\varrho^0) n^{1-k}\sigma_{k-1}^{\overline{\psi},\overline{\varrho}}\left(\frac{n}{t}\right).$$
Then, the series \eqref{eq-Four-exp-mock-G} defines a mock modular form of weight $2-k$, level $tLM$ and character $\overline{\psi\varrho}$ with shadow $E_k^{\psi,\varrho,t}(z)$.
\end{theorem}

\begin{theorem}\label{teo-fourier-exp-mock-2,I,I}
Assume $k=2$ and $(\psi,\varrho)= (\bfo_L,\bfo_M)$.
Put
$$c_{0}^+(0,t,\bfo_L,\bfo_M):=\frac{\pi\log(t)}{M}\prod_{p\mid L}(1-p^{-1}) \prod_{p\mid M}(1-p^{-1}),$$
and for $n>0$ put
$$c_{0}^+(n,t,\bfo_L,\bfo_M):=\frac{\pi}{M^2}n^{-1}\left(\sigma_1^{\bfo_L,\bfo_M}(n)-t\sigma_1^{\bfo_L,\bfo_M}\left(\frac{n}{t}\right)\right).$$
Then, the series \eqref{eq-Four-exp-mock-G} defines a mock modular form of weight $0$, level $tLM$ and trivial character with shadow $E_2^{\bfo_L,\bfo_M,t}(z)$.
\end{theorem}

\begin{theorem}\label{teo-four-exp-1}
Assume $k=1$. 
Put
$$c_{1}^+(0,t,\psi,\varrho):=\left\{
\begin{array}{ll}
2L'(1,\varrho), & \mbox{ if } \psi=\bfo_1,\\
2\pi i\big( \log(2tM)L(0,\overline{\psi})-L'(0,\overline{\psi}) \big)   \prod_{p\mid M} (1-p^{-1}), & \mbox{ if }\varrho=\bfo_M, \\
0, & \mbox{ otherwise},
\end{array}
\right.$$
and for $n>0$ put
\begin{eqnarray*}
c_{1}^+(n,t,\psi,\varrho)&:=&
-\frac{2\pi i }{M}W(\varrho^0) \bigg(\sigma_0^{\overline{\psi},\overline{\varrho}}\left(\frac{n}{t}\right)\left(\log\left(\frac{\pi}{M^2}\right)+\gamma-\log(n)\right)\\
& & +2 \sum_{0<c\mid \frac{n}{t}}\overline{\psi}\left(\frac{n}{tc}\right)\log(c)
\sum_{0<d\mid \gcd(\ell_{\varrho},c)}d\, \mu\left(\frac{\ell_{\varrho}}{d}\right)
\varrho^0\left(\frac{\ell_{\varrho}}{d}\right)\overline{\varrho^0}\left(\frac{c}{d}\right)\bigg),
\end{eqnarray*}
where the last sum is defined to be zero if $t\nmid n$. Then, the series \eqref{eq-Four-exp-mock-G} defines a mock modular form of weight $1$, level $tLM$ and character $\overline{\psi\varrho}$ with shadow $E_1^{\psi,\varrho,t}(z)$.
\end{theorem}

As applications of these results, we will construct mock modular forms whose shadows are Hecke's Eisenstein series associated to imaginary quadratic fields, giving an alternative proof of some results by Kudla, Rapoport and Yang (see \cite{KRY99}, Theorem 1), and Schofer (see \cite{SCH}, Theorem 4.1). We will also construct forms whose shadows are particular examples of theta series, such as 
$\Theta^2(z)$, $\Theta^4(z)$, $\Theta^6(z)$ and $\Theta^8(z)$. By a result of Rankin (see \cite{RAN}) these are the only even powers of $\Theta(z)$ that are linear combinations of Eisenstein series.

\subsection{Organization of the article} This paper is organized as follows: In Section \ref{section-hol-Eis} we give the definition of the Eisenstein series $E_k^{\psi,\varrho,t}(z)$ mentioned above. In Section \ref{section-non-hol-Eis} we recall the definition and basic properties of some non-holomorhic Eisenstein series which play a central role in the construction of the harmonic pre-image of $E_k^{\psi,\varrho,t}(z)$, which is given in Section \ref{sect-const-preimage}. Section \ref{sect-four-expansions} is devoted to the proofs of Theorems  \ref{teo-fourier-exp-mock-geq2}, \ref{teo-fourier-exp-mock-2,I,I} and \ref{teo-four-exp-1}. Finally, in Section \ref{sec-examples} we give some examples including the ones mentioned above.

\section{Holomorphic Eisenstein series}\label{section-hol-Eis}

In this section we review the construction of the holomorphic Eisenstein series 
$E_k^{\psi,\varrho,t}(z)$ mentioned in the introduction. Throughout this and the next sections,
we let $N$, $k$, $L$ and $M$ be integers with $N, k, L, M\geq 1$.
Moreover, we let $\psi$ and $\varrho$ denote Dirichlet characters mod $L$ and $M$, respectively, satisfying \eqref{eq-psi-rho-k-compatible}. We set
$$
C_k(\psi,\varrho)
:=\left\{
\begin{array}{ll}
L(k,\conj{\varrho}), 
& \mbox{ if }\psi= \bfo_1,\\
0, & \mbox{ if }\psi\neq \bfo_1,\\
\end{array}
\right.
$$
and we define
$$D(\psi,\varrho):=\left\{
\begin{array}{ll}
-\pi i L(0,\psi) \prod_{p\mid M}(1-p^{-1}), & \mbox{ if }\varrho= \bfo_M,\\
0, & \mbox{ if }\varrho \neq \bfo_M.
\end{array}
\right.$$
With the above notation and with $\sigma^{\psi,\varrho}_{k-1}(n)$ given by \eqref{def_sigma_psi_rho}, we define the functions $E^{\psi,\varrho}_k(z)$ by their Fourier expansion
as follows.
When $k>2$, or $k=2$ and $(\psi,\varrho)\neq (\bfo_L,\bfo_M)$, we define
$$
E^{\psi,\varrho}_k(z):=C_k(\psi,\varrho)+\left(-\frac{2\pi i}{M}\right)^k \frac{W(\conj{\varrho^0})}{(k-1)!}\sum_{n=1}^{\infty}\sigma^{\psi,\varrho}_{k-1}(n)q^n.$$
When $k=2$ and $(\psi,\varrho)=(\bfo_L,\bfo_M)$, we define
$$E_2^{\bfo_L,\bfo_M}(z):=C_2(\bfo_L,\bfo_M)-\frac{\pi}{2 M y}\prod_{p\mid M}(1-p^{-1})\prod_{p\mid L}(1-p^{-1})-\frac{4\pi^2 }{M^2} \sum_{n=1}^{\infty}\sigma^{\bfo_L,\bfo_M}_{1}(n)q^n.$$
When $k=1$, we define
$$E^{\psi,\varrho}_1(z):=C_1(\psi,\varrho)+D(\psi,\varrho)-\frac{2\pi i }{M} W(\overline{\varrho^0})\sum_{n=1}^{\infty}\sigma^{\psi,\varrho}_{0}(n)q^n.$$
Note that in this case $\psi(-1)\varrho(-1)=-1$, hence if $\psi=\bfo_1$, then $\varrho\neq \bfo_M$ and  $C_1(\psi,\varrho)$ is a well-defined complex number. 

The following result is well-known 
and it is also a consequence of Proposition \ref{prop-special-value} in Section \ref{section-non-hol-Eis}.

\begin{proposition}\label{prop-Eis-series-hol}
If $k\neq 2$, or $k=2$ and $(\psi,\varrho)\neq (\bfo_L,\bfo_M)$, then the series $E^{\psi,\varrho}_1(z)$ defines a holomorphic modular form in $M_k(LM,\psi\varrho)$. Moreover, the series $E_2^{\bfo_L,\bfo_M}(z)$ defines a non-holomorphic function transforming like a form in $M_1(LM,\bfo_{LM})$.
\end{proposition}

For an integer $t\geq 1$, we define the Eisenstein series
$$
E^{\psi,\varrho,t}_k(z):=\left\{\begin{array}{ll}
E_2^{\bfo_L,\bfo_M}(z)-tE_2^{\bfo_L,\bfo_M}(tz), & \mbox{ if }k=2 \mbox { and } (\psi,\varrho)=(\bfo_L,\bfo_M),\\
E^{\psi,\varrho}_k(tz), & \mbox{ otherwise}.
\end{array}
\right.$$
It follows from Proposition \ref{prop-Eis-series-hol} that $E^{\psi,\varrho,t}_k(z)$ is a holomorphic modular form in $M_k(tLM,\psi\varrho)$. The following result is also well-known.

\begin{proposition}\label{prop-Eis-orth-compl}
The Eisenstein series $E^{\psi,\varrho,t}_k(z)$ belongs to the subspace $N_k(tLM,\psi \varrho)$. Moreover, given 
a Dirichlet character $\chi$ mod $N$ with $\chi(-1)=(-1)^k$, the collection of Eisenstein series $E^{\psi,\varrho,t}_k(z)$, where $t\geq 1$ is an integer, and $\psi$, $\varrho$ are Dirichlet characters mod $L$ and $M$, respectively, satisfying $tLM\mid N$ and $\psi \varrho=\chi$, generates $N_k(N,\chi)$.
\end{proposition}

\begin{remark}
We remark that in the case $k\geq 2$ one obtains a basis of $N_k(N,\chi)$ by restricting to Eisenstein series $E^{\psi,\varrho,t}_k(z)$ with $\psi$ and $\varrho$ primitive, and with $t\geq 2$ if $k=2$ and $(\psi,\varrho)=(\bfo_L,\bfo_M)$. In the case $k=1$ one has to consider unordered pairs of primitive characters $\{\psi,\varrho\}$, since $E^{\psi,\varrho,t}_1(z)$ and $E^{\varrho,\psi,t}_1(z)$ satisfy the identity
$$\frac{M}{W\left(\overline{\varrho}\right)}E^{\psi,\varrho,t}_1(z)=\frac{L}{W\left(\overline{\psi}\right)}E^{\varrho,\psi,t}_1(z).$$
This identity follows easily from comparison of Fourier expansions, together with the functional equation of Dirichlet $L$-series for primitive characters. 
\end{remark}

We refer the reader to \cite{DiSh} (see Theorems 4.5.2, 4.6.2 and 4.8.1) for a nice exposition on the results described in this section. Note that the Eisenstein series are normalized in a slightly different way in \begin{it}loc.~cit\end{it}.

\section{Non-holomorphic Eisenstein series}\label{section-non-hol-Eis}

In this section we consider non-holomorphic Eisenstein series that evaluated at special points 
equal $E^{\psi,\varrho}_k(z)$. These non-holomorphic forms were studied by Hecke and are considered nowadays as part of the classical theory of automorphic forms. 
We let $k,L,M,\psi,\varrho$ be defined as in Section \ref{section-hol-Eis}.
%
For $s\in \C$ with $k+2\Re(s)>2$, we define the Eisenstein series
\begin{equation}\label{def_Ekzspsirho}
E_k(z,s,\psi,\varrho):=\frac{1}{2}\sum_{\substack{m,n\in \Z\\ (m,n)\neq (0,0)}}\psi(m)\varrho(n)(mz+n)^{-k}|mz+n|^{-2s}.
\end{equation}
A modern treatment of these Eisenstein series can be found in \cite{MIY}, \S 7.2 (more precisely, the function studied in \begin{it}loc.~cit.\end{it}~equals $2E_k(z,s,\psi,\varrho)$ in our notation). This Eisenstein series converges uniformly and absolutely for $s\in \C$ with $2\Re(s)+k\geq 2+\varepsilon$, for any $\varepsilon>0$. 
Letting
$$
\Gamma_0(L,M):=\left\{\pm \begin{pmatrix}a&b\\ c&d\end{pmatrix}\in \mathrm{PSL}_2(\Z): c=0\, (\mbox{mod }L),b=0\, (\mbox{mod }M) \right\},
$$
they satisfy  the transformation formula
$$E_k(\gamma z,s,\psi,\varrho)=\psi(d)\overline{\varrho(d)}(cz+d)^k|cz+d|^{2s}E_k(z,s,\psi,\varrho)$$
for any
$\gamma=\pm \bigl(\begin{smallmatrix}a&b\\ c&d\end{smallmatrix}\bigr)\in \Gamma_0(L,M)$.
In particular, the function
\begin{equation}\label{def-non-hol-Eis-series}
y^{s}E_k(z,s,\psi,\varrho)=\frac{1}{2}L(\psi\varrho,2s+k)\sum_{\substack{c,d\in \Z\\ \gcd(c,d)=1}}\psi(c)\varrho(d)(cz+d)^{-k}\Im\left(\begin{pmatrix}\ast &\ast \\ c&d\end{pmatrix} z\right)^s  
\end{equation}
(here $\left(\begin{smallmatrix}\ast &\ast \\ c&d\end{smallmatrix}\right)$ denotes any matrix in $\mathrm{SL}_2(\Z)$ with bottom row $(c\ d)$) transforms like a modular form of weight $k$ and character $\psi \overline{\varrho}$ for $\Gamma_0(L,M)$, and
$$y^{s}E_k(Mz,s,\psi,\varrho)$$
transforms like a modular form of weight $k$, level $LM$ and character $\psi \overline{\varrho}$. 

For our purpose, the Fourier expansion of $y^{s}E_k(Mz,s,\psi,\varrho)$ at the cusp $\infty$ will be very useful. To state it, we introduce the following notation. We set
\begin{align*}
A_k(s;\varrho)
&:=\frac{2^{k}\pi^{s+k}i^{-k}}{M^{s+k}\Gamma(s+k)}W(\varrho^0),\\
B_k(s;\varrho)
&:=\frac{2^{-k}\pi^si^{-k}\varrho(-1)}{M^{s}\Gamma(s)}W(\varrho^0),\\
C_k(s;\psi,\varrho)
&:=\left\{
\begin{array}{ll}
L(2s+k,\varrho), & \mbox{ if }\psi= \bfo_1,\\
0, & \mbox{ if }\psi\neq \bfo_1,
\end{array}
\right.
\end{align*}
where $\Gamma(s)$ denotes Euler's $\Gamma$-function, and 
\begin{align*}
D_k(s;\psi,\varrho)&:=\left\{
\begin{array}{ll}
\sqrt{\pi}i^{-k}\frac{\Gamma\left(\frac{2s+k-1}{2}\right)\Gamma\left(\frac{2s+k}{2}\right)}{\Gamma(s)\Gamma(s+k)}L(2s+k-1,\psi)\prod_{p\mid M}(1-p^{-1}), & \mbox{ if }\varrho= \bfo_M,\\
0, & \mbox{ if }\varrho \neq \bfo_M.
\end{array}
\right.
\end{align*}
Furthermore, we define
$$a_k(s;n,\psi,\varrho):=\sum_{0<c\mid n}\psi\left(\frac{n}{c}\right)c^{2s+k-1}\sum_{0<d\mid \gcd(\ell_{\varrho},c)}d\, \mu\left(\frac{\ell_{\varrho}}{d}\right)
\varrho^0\left(\frac{\ell_{\varrho}}{d}\right)\overline{\varrho^0}\left(\frac{c}{d}\right),$$
where $\mu(n)$ denotes the M\"obius function. For $y>0$ and $\alpha, \beta\in \C$ with $\Re(\beta)>0$, we also define
\begin{align}\label{def_omega}
\omega(y;\alpha,\beta):=\frac{y^{\beta}}{\Gamma(\beta)}\int_0^{\infty}e^{-yu}(u+1)^{\alpha-1}u^{\beta-1}du.
\end{align}
This function is basically equal to a confluent hypergeometric function and can be extended 
to an analytic function on $\{(z;\alpha,\beta)\in \C^3: \Re(z)>0\}$, satisfying the identities
(see \cite{MIY}, Theorem 7.2.4, p.~279, and Lemma 7.2.6, p.~281)
\begin{align}
\omega(z;1-\beta,1-\alpha)&=\omega(z;\alpha,\beta),\label{eq-omega-funct-equation}\\
\omega(z;\alpha,0)&=1.\label{eq-omega-zero}
\end{align}

With the above notation, we have the following Fourier expansion (see, e.g.,~\cite{MIY}, Theorem 7.2.9, p.~284)
\begin{eqnarray}\label{FEnhEis}
 & & y^{s}E_k(Mz,s,\psi,\varrho)\\
 &=& y^sC_k(s;\psi,\varrho)+y^{-s-k+1}M^{-2s-k+1}D_k(s;\psi,\varrho) \nonumber \\
& & +A_k(s;\varrho)M^{-s}\sum_{n=1}^{\infty}a_k(s;n,\psi,\varrho)n^{-s}e^{2\pi i nz}\omega(4\pi ny;k+s,s) \nonumber \\
& & +y^{-k}B_k(s;\varrho)M^{-k-s}\sum_{n=1}^{\infty}a_k(s;n,\psi,\varrho)n^{-s-k}e^{-2\pi i n \conj{z}}\omega(4\pi ny ;s,k+s). \nonumber
\end{eqnarray}
This Fourier expansion can be used to meromorphically continue the function $y^{s}E_k(Mz,s,\psi,\varrho)$ to all $s\in \C$. In particular, the following assertions hold.

\begin{proposition}\label{prop-special-value}
The function $y^sE_k(Mz,s,\psi,\varrho)$ is analytic at $s=0$ and we have
\begin{equation}\label{eq.special-value}
E_k(Mz,0,\psi,\overline{\varrho})=E_k^{\psi,\varrho}(z).
\end{equation}
\end{proposition}
\begin{proof}
For the first assertion, we refer the reader to \cite{MIY}, Corollary 7.2.10, p.~285. The proof of equality \eqref{eq.special-value} consists in a simple comparison of Fourier expansions. 
Indeed, one easily proves the following identities
\begin{align*}
C_k(0;\psi,\overline{\varrho})&= C_k(\psi,\varrho),\\
D_k(0;\psi,\overline{\varrho})&=
\left\{\begin{array}{ll}
-\frac{\pi}{2} \prod_{p\mid L} (1-p^{-1})\prod_{p \mid M}(1-p^{-1}), & \mbox{ if }k=2, (\psi,\varrho)=(\bfo_L,\bfo_M),\\[0.2cm]
D(\psi,\varrho), & \mbox{ if }k=1,\\
0, & \mbox{ otherwise},
\end{array}\right.
 \\
A_k(0;\overline{\varrho}) &= \left(-\frac{2\pi i}{M}\right)^k \frac{W(\overline{\varrho^0})}{(k-1)!},\\
B_k(0;\overline{\varrho}) &= 0,\\
a_k(0;n,\psi,\overline{\varrho}) &= \sigma_{k-1}^{\psi,\varrho}(n).
\end{align*}
This together with identity \eqref{eq-omega-zero} proves that the Fourier expansion of $E_k(Mz,0,\psi,\overline{\varrho})$ given by \eqref{FEnhEis} and the Fourier expansion defining $E_k^{\psi,\varrho}(z)$ agree, and henceforth these two functions are equal. This completes the proof.
\end{proof}

\section{Construction of pre-images}\label{sect-const-preimage} 

Let $k,L,M,\psi,\varrho$ be defined as in Section \ref{section-hol-Eis} and let $t\geq 1$ be an integer. 
In this section we construct a harmonic Maass form $\E_{2-k}^{\psi,\varrho,t}\in H_{2-k}(tLM,\overline{\psi\varrho})$, whose image under the differential operator $\xi_{2-k}$ equals $E_k^{\psi,\varrho,t}(z)$.  

When $k> 2$, the point $s=k-1$ is in the domain of the definition \eqref{def_Ekzspsirho} of $E_{2-k}(tMz,s,\overline{\psi},\varrho)$ and  we can define
$$\E^{\psi,\varrho,t}_{2-k}(z):=(k-1)^{-1}y^{k-1}E_{2-k}(tMz,k-1,\overline{\psi},\varrho).$$
When $k=2$ and $(\psi,\varrho)\neq (\bfo_L,\bfo_M)$, the meromorphic continuation of the function $E_{2-k}(tMz,s,\overline{\psi},\varrho)$ is analytic at $s=1$ (see \cite{MIY}, Corollary 7.2.11, p.~286) and we can define
$$\E^{\psi,\varrho,t}_{0}(z):=yE_{0}(tMz,1,\overline{\psi},\varrho).$$
When $k=2$ and $(\psi,\varrho)= (\bfo_L,\bfo_M)$ the function $y^sE_{0}(tMz,s,\bfo_L,\bfo_M)$ has a simple pole at $s=1$ with residue
$$\frac{\pi}{2tM}\prod_{p\mid L }(1-p^{-1})\prod_{p\mid M}(1-p^{-1})$$
(\begin{it}loc.~cit.\end{it}, Corollary 7.2.10 part (3) in p.~285), hence  we can define
$$\E^{\bfo_L,\bfo_M,t}_{0}(z):=\lim_{s\to 1}\left(y^sE_{0}(Mz,s,\bfo_L,\bfo_M)-ty^sE_{0}(tMz,s,\bfo_L,\bfo_M)\right).$$
When $k=1$, the meromorphic continuation of $E_1(Mz,s,\overline{\psi},\varrho)$ is analytic at $s=0$ (\begin{it}loc.~cit.\end{it}, Corollary 7.2.11 in p.~286) and we can define
$$\E^{\psi,\varrho,t}_{1}(z):=\frac{\partial}{\partial s}\left(y^{s}E_{1}(tMz,s,\overline{\psi},\varrho)\right)\big|_{s=0}.$$

With the above definitions, we now prove the following result.

\begin{theorem}\label{teo-preimage}
We have $\E^{\psi,\varrho,t}_{2-k}(z)\in H_{2-k}(LM,\overline{\psi\varrho})$ and $\xi_{2-k}(\E^{\psi,\varrho,t}_{2-k})=E^{\psi,\varrho,t}_{k}$.
\end{theorem}

\begin{proof} We start by proving the identity $\xi_{2-k}(\E^{\psi,\varrho,t}_{2-k})=E^{\psi,\varrho,t}_{k}$. Let us first assume that $k>2$, or $k=2$ and $(\psi,\varrho)\neq (\bfo_L,\bfo_M)$. Combining the definition of $\E^{\psi,\varrho,t}_{2-k}$ with the Fourier expansion of $E_k(tMz,s,\psi,\overline{\varrho})$ given by \eqref{FEnhEis}, we deduce
\begin{eqnarray*}
 & & (k-1)\E^{\psi,\varrho,t}_{2-k}(z)\\
 &=& y^{k-1}C_{2-k}(k-1;\overline{\psi},\varrho)
 +(tM)^{1-k}D_{2-k}(k-1;\overline{\psi},\varrho) \nonumber \\
& & +A_{2-k}(k-1;\varrho)(tM)^{1-k}\sum_{n=1}^{\infty}a_{2-k}(k-1;n,\overline{\psi},\varrho)n^{1-k}e^{2\pi i tnz}\omega(4\pi tn y ;1,k-1) \nonumber \\
& & +y^{k-2}B_{2-k}(k-1;\varrho)(tM)^{-1}\sum_{n=1}^{\infty}a_{2-k}(k-1;n,\overline{\psi},\varrho)n^{-1}e^{-2\pi i tn \conj{z}}\omega(4\pi tn y ;k-1,1). \nonumber
\end{eqnarray*}
Now, by \eqref{eq-omega-funct-equation} and \eqref{eq-omega-zero}, we have
\begin{equation}\label{eq-omega-k-1}
\omega(4\pi tn y ;1,k-1) =\omega(4\pi tn y ;2-k,0)=1.
\end{equation}
Furthermore, substituting $u:=r-1$ in the definition \eqref{def_omega} of $\omega(4\pi tn y ;k-1,1)$, we obtain
\begin{eqnarray}\label{eq-omega-k-2}
y^{k-2}\omega(4\pi tn y ;k-1,1)
&=& 4\pi tny^{k-1} e^{4\pi  tn y}\int_1^{\infty} e^{-4\pi tnyr}r^{k-2}dr   \\
&=& 4\pi tne^{4\pi tn y}\beta_{2-k}(tn,y), \nonumber
\end{eqnarray}
where $\beta_{2-k}(tn,y)$ is defined in \eqref{def_beta1}.
Recalling that $y^{k-1}=(1-k)\beta_{2-k}(0,y)$ by $\eqref{def_beta2}$, this gives
\begin{eqnarray}\label{eq-Fou-exp-preimag-kgeneral}
 & & (k-1)\E^{\psi,\varrho,t}_{2-k}(z)\\
 &=& (1-k)\beta_{2-k}(0,y)C_{2-k}(k-1;\overline{\psi},\varrho)+(tM)^{1-k}D_{2-k}(k-1;\overline{\psi},\varrho) \nonumber \\
& & +A_{2-k}(k-1;\varrho)(tM)^{1-k}\sum_{n=1}^{\infty}a_{2-k}(k-1;n,\overline{\psi},\varrho)n^{1-k}q^{tn} \nonumber \\
& & +4\pi B_{2-k}(k-1;\varrho)M^{-1}\sum_{n=1}^{\infty}a_{2-k}(k-1;n,\overline{\psi},\varrho)\beta_{2-k}(tn,y)q^{-tn}. \nonumber
\end{eqnarray}
Since
\begin{equation}\label{eq-xi-beta-exp}
\xi_{2-k}(\beta_{2-k}(n,y)q^{-n})=-q^n \mbox{ for every integer }n\geq 0,
\end{equation}
and since $\xi_{2-k}$ annihilates holomorphic functions, we conclude from \eqref{eq-Fou-exp-preimag-kgeneral} that
\begin{eqnarray*}
\xi_{2-k}\left(\E^{\psi,\varrho,t}_{2-k}\right)(z)
 &=& \overline{C_{2-k}(k-1;\overline{\psi},\varrho)} \nonumber \\
& & -4\pi (k-1)^{-1}\overline{B_{2-k}(k-1;\varrho)}M^{-1}\sum_{n=1}^{\infty}\overline{a_{2-k}(k-1;n,\overline{\psi},\varrho)}q^{tn}. \nonumber
\end{eqnarray*}
Finally, using the definition of the Fourier coefficients given in Sections \ref{section-hol-Eis} and \ref{section-non-hol-Eis}, we get the identities
\begin{eqnarray*}
 \overline{C_{2-k}(k-1;\overline{\psi},\varrho)} &=& C_k(\psi,\varrho),\\
 -4\pi (k-1)^{-1}\overline{B_{2-k}(k-1;\varrho)}M^{-1}&=&\left(-\frac{2\pi i}{M}\right)^k \frac{ W(\overline{\varrho^0})}{(k-1)!},\\
 \overline{a_{2-k}(k-1;n,\overline{\psi},\varrho)}&=&  \sigma_{k-1}^{\psi,\varrho}(n) ,
\end{eqnarray*}
which can easily be checked case by case. All in all, this proves the equality $\xi_{2-k}(\E^{\psi,\varrho,t}_{2-k})=E_{k}^{\psi,\varrho,t}$, as asserted.\\
Let us now assume that $k=2$ and $(\psi,\varrho)=(\bfo_L,\bfo_M)$. Then, the function $y^sE_0(tMz,s,\bfo_L,\bfo_M)$ has a simple pole at $s=1$ coming from the second term in the Fourier expansion \eqref{FEnhEis}. We have
\begin{eqnarray*}
 & & \lim_{s\to 1 } \left(y^{1-s}M^{1-2s}D_0(s;\bfo_L,\bfo_M)-ty^{1-s}(tM)^{1-2s}D_0(s;\bfo_L,\bfo_M)\right) \\
 &=& \frac{\pi}{M}  \lim_{s\to 1}(1-t^{2-2s})L(2s-1,\bfo_L) \prod_{p\mid M}(1-p^{-1})\\
 &=& \frac{\pi}{M}  \lim_{s\to 1}(1-t^{2-2s})\zeta(2s-1) \prod_{p\mid L}(1-p^{-1}) \prod_{p\mid M}(1-p^{-1})\\
 &=& \frac{\pi}{M}\pi \log(t)\prod_{p\mid L}(1-p^{-1}) \prod_{p\mid M}(1-p^{-1}).
\end{eqnarray*}
Using the Fourier expansion \eqref{FEnhEis} together with the identities \eqref{eq-omega-k-1} and \eqref{eq-omega-k-2}, this yields the Fourier expansion
\begin{eqnarray*}
  \E_0^{\bfo_L,\bfo_M,t}(z)
 &=& -(1-t)C_0(1;\bfo_L,\bfo_M)\beta_0(0,y)+\frac{\pi}{M} \log(t)\prod_{p\mid L}(1-p^{-1}) \prod_{p\mid M}(1-p^{-1})  \\
& & +A_0(1;\bfo_M)\frac{1}{M}\sum_{n=1}^{\infty}a_0(1;n,\bfo_L,\bfo_M)n^{-1}(q^n-q^{tn}) \nonumber \\
& & + B_0(1;\bfo_M)\frac{4\pi}{M}\sum_{n=1}^{\infty}a_0(1;n,\bfo_L,\bfo_M)(q^{-n}\beta_0(n,y)-tq^{-tn}\beta_0(tn,y)). \nonumber
\end{eqnarray*}
Since $C_0(1;\bfo_L,\bfo_M)=C_2(\bfo_L,\bfo_M)$, $A_0(1;\bfo_M)=B_0(1;\bfo_M)=\frac{\pi}{M}$ and $a_0(1;n,\bfo_L,\bfo_M)=\sigma_1^{\bfo_L,\bfo_M}(n)$, we get
\begin{eqnarray}\label{eq-Fou-exp-preimage-wt2-triv-char}
 \hspace{0.5cm} \E_0^{\bfo_L,\bfo_M,t}(z)
 &=& -(1-t)C_2(\bfo_L,\bfo_M) \beta_0(0,y)+\frac{\pi}{M} \log(t)\prod_{p\mid L}(1-p^{-1}) \prod_{p\mid M}(1-p^{-1})  \\
& & +\frac{\pi}{M^2}\sum_{n=1}^{\infty}\sigma_1^{\bfo_L,\bfo_M}(n)n^{-1}(q^n-q^{tn}) \nonumber \\
& & +\frac{4\pi^2}{M^2}\sum_{n=1}^{\infty}\sigma_1^{\bfo_L,\bfo_M}(n)(q^{-n}\beta_0(n,y)-tq^{-tn}\beta_0(tn,y)). \nonumber
\end{eqnarray}
As before, we use \eqref{eq-xi-beta-exp} and obtain
\begin{equation*}
 \xi_0\left( \E_0^{\bfo_L,\bfo_M,t}\right)(z)
 =(1-t)C_2(\bfo_L,\bfo_M) -\frac{4\pi^2}{M^2}\sum_{n=1}^{\infty}\sigma_1^{\bfo_L,\bfo_M}(n)(q^{n}-tq^{tn}).
\end{equation*}
This proves that $\xi_0( \E_0^{\bfo_L,\bfo_M,t})=E^{\psi,\varrho,t}_2$, as claimed.\\
Finally, let us assume that $k=1$. Then, $\E^{\psi,\varrho,t}_{1}(z)=\frac{\partial}{\partial s}\left(y^{s}E_{1}(tMz,s,\overline{\psi},\varrho)\right)\big|_{s=0}$ and we start by deducing from  \eqref{FEnhEis} the following Fourier expansion 
\begin{eqnarray}\label{eq-four-exp-preimage-wt1}
 & & y^{s}E_1(tMz,s,\overline{\psi},\varrho)\\
 &=& y^sC_1(s;\overline{\psi},\varrho)+(yt^2M^2)^{-s}D_1(s;\overline{\psi},\varrho) \nonumber \\
& & +A_1(s;\varrho)(tM)^{-s}\sum_{n=1}^{\infty}a_1(s;n,\overline{\psi},\varrho)n^{-s}e^{2\pi i tnz}\omega(4\pi tn y ;1+s,s) \nonumber \\
& & +y^{-1}B_1(s;\varrho)(tM)^{-s-1}\sum_{n=1}^{\infty}a_1(s;n,\overline{\psi},\varrho)n^{-s-1}e^{-2\pi i tn \conj{z}}\omega(4\pi tn y ;s,1+s). \nonumber
\end{eqnarray}
We have the identity
\begin{eqnarray}\label{eq-derivative-C}
\frac{\partial}{\partial s}\left(y^sC_1(s;\overline{\psi},\varrho)\right)\big|_{s=0}&=&
\left\{
\begin{array}{ll}
\log(y)L(1,\varrho)+2L'(1,\varrho), & \mbox{ if }\psi=\bfo_1,\\
0, & \mbox{ if }\psi \neq \bfo_1,
\end{array}\right..
\end{eqnarray}
Next, we compute
\begin{eqnarray*}
\frac{\partial}{\partial s}\left((yt^2M^2)^{-s}D_1(s;\overline{\psi},\varrho)\right)\big|_{s=0}&=&
\sqrt{\pi}i\Bigl( L(0,\overline{\psi})\left(\sqrt{\pi}\log(yt^2M^2)-\Gamma'(\tfrac{1}{2})+\sqrt{\pi}\Gamma'(1)\right)\\
& & -\sqrt{\pi}2L'(0,\overline{\psi}) \Bigr)\prod_{p\mid M}(1-p^{-1}).
\end{eqnarray*}
Using the well-known identities
\begin{equation}\label{eq-Gamma-derivative-values}
\Gamma'(\tfrac{1}{2})=-\sqrt{\pi}(2\log(2)+\gamma)\quad \mbox{ and }\quad \Gamma'(1)=-\gamma,
\end{equation}
where $\gamma$ denotes Euler's constant (see, e.g., \cite{GRA}, formulas 8.366-1 and 8.366-2), we get
\begin{eqnarray}\label{eq-derivative-D}
& & \frac{\partial}{\partial s}\left((yt^2M^2)^{-s}D_1(s;\overline{\psi},\bfo_M)\right)\big|_{s=0}\\
&=&
\pi i\Bigl( L(0,\overline{\psi})\left(\log(yt^2M^2)+2\log(2)\right) -2L'(0,\overline{\psi}) \Bigr)\prod_{p\mid M}(1-p^{-1}), \nonumber
\end{eqnarray}
while
\begin{equation}\label{eq-derivative-D-trivial}
\frac{\partial}{\partial s}\left((yt^2M^2)^{-2s}D_1(s;\overline{\psi},\varrho)\right)\big|_{s=0}=0, \quad\mbox{ if }\varrho\neq \bfo_M. 
\end{equation}
Now, by \eqref{eq-omega-funct-equation}, we have that $\omega(4\pi tn y ;1+s,s)$ is invariant under $s\mapsto -s$, hence its derivative with respect to $s$ vanishes at $s=0$. Therefore, using \eqref{eq-omega-zero} we conclude that
\begin{equation}\label{eq-detivative-A}
\frac{\partial}{\partial s}\left(A_1(s;\varrho)a_1(s;n,\overline{\psi},\varrho)(tMn)^{-s}\omega(4\pi tn y ;1+s,s)\right)\big|_{s=0}
\end{equation}
equals
$$\frac{\partial}{\partial s}\left(A_1(s;\varrho)a_1(s;n,\overline{\psi},\varrho)(tMn)^{-s}\right)\big|_{s=0}.$$
Using the definitions of $A_1(s;\varrho)$ and $a_1(s;n,\overline{\psi},\varrho)$ given in Section \ref{section-non-hol-Eis}, we get that \eqref{eq-detivative-A} equals
\begin{eqnarray}\label{eq-derivative-A-formula}
 & & -\frac{2\pi i}{M}  W(\varrho^0)\bigg( \sigma_0^{\overline{\psi},\overline{\varrho}}(n) \left(\log\left(\frac{\pi}{M}\right)-\Gamma'(1)-\log(tMn)\right) \\
 & &+2 \sum_{0<c\mid n}\overline{\psi}\left(\frac{n}{c}\right)\log(c)
\sum_{0<d\mid \gcd(\ell_{\varrho},c)}d\, \mu\left(\frac{\ell_{\varrho}}{d}\right)
\varrho^0\left(\frac{\ell_{\varrho}}{d}\right)\overline{\varrho^0}\left(\frac{c}{d}\right) \bigg).\nonumber
\end{eqnarray}
Furthermore, one has that
\begin{equation}\label{eq-derivative-B}
\frac{\partial}{\partial s}\left(y^{-1}B_1(s;\varrho)a_1(s;n,\overline{\psi},\varrho)(tMn)^{-1-s}\omega(4\pi tn y ;s,1+s)\right)\big|_{s=0}
\end{equation} 
equals
$$\left(\frac{\partial}{\partial s}B_1(s;\varrho)\big|_{s=0}\right)a_1(0;n,\overline{\psi},\varrho)(tMny)^{-1}\omega(4\pi tn y ;0,1),$$
since $B_1(0;\varrho)=0$. We conclude, using the definitions of $B_1(s;\varrho)$ and $a_1(0;n,\overline{\psi},\varrho)$ given in Section \ref{section-non-hol-Eis}, together with \eqref{eq-omega-k-2}, that \eqref{eq-derivative-B}  equals
\begin{eqnarray}\label{eq-derivative-B-formula}
-\frac{2\pi i}{M}\varrho(-1)W(\varrho^0)\sigma_0^{\overline{\psi},\overline{\varrho}}(n)\beta_1(tn,y)e^{4\pi t n y}.
\end{eqnarray}
Summing up, using \eqref{eq-xi-beta-exp}, the identity $\varrho(-1)\overline{W(\varrho^0)}=W(\overline{\varrho^0})$, and recalling that $\log(y)=-\beta_1(0,y)$ by \eqref{def_beta2}, we obtain
\begin{equation*}
 \xi_1\left(\E_1^{\psi,\varrho,t}\right)(z)= C_1(\psi,\varrho)+D(\psi,\varrho) -\frac{2\pi i}{M}W(\overline{\varrho^0})\sum_{n=1}^{\infty}\sigma_0^{\psi,\varrho}(n)q^{tn}.
\end{equation*} 
This proves that $\xi_1(\E_1^{\psi,\varrho,t})=E_1^{\psi,\varrho,t}$, as asserted.\\
We have thus shown that $\xi_{2-k}(\E_{2-k}^{\psi,\varrho,t})=E_k^{\psi,\varrho,t}$ in all possible cases. This automatically implies that $\E_{2-k}^{\psi,\varrho,t}(z)$ is harmonic. Indeed, by \eqref{eq-Delta-xi} we have
$$\Delta(\E_{2-k}^{\psi,\varrho,t})=\xi_k\circ \xi_{2-k}\left(\E_{2-k}^{\psi,\varrho,t}\right)=\xi_k\left(E_k^{\psi,\varrho,t}\right)=0.$$
The fact that $\E_{2-k}^{\psi,\varrho,t}(z)$ transforms like a modular form in $M_{2-k}(LM,\overline{\psi \varrho})$ follows from the modularity of the function $y^{s}E_{2-k}(tMz,s,\overline{\psi},\varrho)$. Finally, since $E_k^{\psi,\varrho,t}$ is holomorphic at every cusp, we conclude that $\E_{2-k}^{\psi,\varrho,t}(z)\in H_{2-k}(LM,\overline{\psi \varrho})$. This completes the proof of the theorem.
\end{proof}

\begin{remark}
Note that the series \eqref{def-non-hol-Eis-series} converges absolutely and uniformly for $s$ in any fixed compact subset of the half-plane $\Re(s)>1-\frac{k}{2}$. In particular, when $k>2$, one can compute $\xi_{2-k}(\E^{\psi,\varrho,t}_{2-k})$ by applying $\xi_{2-k}$ to each term in this series and then evaluating at $s=k-1$. This argument, however, needs some adjustments in the cases $k=1$ and $k=2$. We avoided this issue by working directly with the Fourier expansions in the proof of Theorem \ref{teo-preimage}. Most of the computations will be re-used in the next section, where we shall give the Fourier expansion of the non-holomorphic part of $\E^{\psi,\varrho,t}_{2-k}(z)$.
\end{remark}

\section{Proof of Theorems  \ref{teo-fourier-exp-mock-geq2}, \ref{teo-fourier-exp-mock-2,I,I} and \ref{teo-four-exp-1}}\label{sect-four-expansions}

In this section we describe the Fourier coefficients in the expansion 
\begin{eqnarray*}
\E^{\psi,\varrho,t}_{2-k}(z)&=&\sum_{n=0}^{\infty}c_{2-k}^+(n,t,\psi,\varrho)q^n -\sum_{n=0}^{\infty}\overline{c_k(n,t,\psi,\varrho)}\beta_{2-k}(n,y)q^{-n}.
\end{eqnarray*}
Note that, by Theorem \ref{teo-preimage}, the coefficients $c_k(n,t,\psi,\varrho)$ are simply the Fourier coefficients of $E^{\psi,\varrho,t}_{k}(z)$, which can be obtained by using the formulas in Section \ref{section-hol-Eis}. For this reason, we shall restrict our attention to the non-holomorphic coefficients $c_{2-k}^+(n,t,\psi,\varrho)$. 

In what follows, we will frequently refer to the computations given in the proof of Theorem \ref{teo-preimage}. In doing so, we only have to take into account the holomorphic terms, since the non-holomorphic terms only contribute to the coefficients $c_{k}(n,t,\psi,\varrho)$. 

\begin{proof}[Proof of Theorem \ref{teo-fourier-exp-mock-geq2}]
By \eqref{eq-Fou-exp-preimag-kgeneral} we have
\begin{eqnarray*}
\sum_{n=0}^{\infty}c_{2-k}^+(n,t,\psi,\varrho)q^n &=&(tM)^{1-k}\bigg( \dfrac{D_{2-k}(k-1;\overline{\psi},\varrho)}{k-1} \\
& & +\frac{A_{2-k}(k-1;\varrho)}{k-1}\sum_{n=1}^{\infty}a_{2-k}(k-1;n,\overline{\psi},\varrho)n^{1-k}q^{tn}\bigg).
\end{eqnarray*}
But
$$\frac{D_{2-k}(k-1;\overline{\psi},\varrho)}{k-1}=\left\{
\begin{array}{ll}
\sqrt{\pi}i^{k-2}\dfrac{\Gamma\left(\frac{k-1}{2}\right)\Gamma\left(\frac{k}{2}\right)}{(k-1)!}L(k-1,\overline{\psi})\prod_{p\mid M}(1-p^{-1}), & \mbox{ if }\varrho= \bfo_M,\\
0, & \mbox{ if }\varrho \neq \bfo_M
\end{array}
\right.$$
By Legendre's duplication formula (see, e.g., \cite{GRA}, formula 8.335-1), we have
$$\Gamma(\tfrac{k-1}{2})\Gamma(\tfrac{k}{2})=2^{2-k}\sqrt{\pi}\Gamma(k-1),$$
hence
$$\frac{D_{2-k}(k-1;\overline{\psi},\varrho)}{k-1}=\left\{
\begin{array}{ll}
\dfrac{ 2^{2-k}\pi i^{k-2}}{k-1}L(k-1,\overline{\psi})\prod_{p\mid M}(1-p^{-1}), & \mbox{ if }\varrho= \bfo_M,\\
0, & \mbox{ if }\varrho \neq \bfo_M
\end{array}
\right.$$
This proves the formula for $c_{2-k}^+(0,t,\psi,\varrho)$. The formula for $c_{2-k}^+(n,t,\psi,\varrho)$ follows from the identities
\begin{align*}
A_{2-k}(k-1;\varrho)&=\frac{2^{2-k}\pi i^{k-2}}{M}W(\varrho^0),\\
a_{2-k}(k-1;n,\overline{\psi},\varrho)&= \sigma_{k-1}^{\overline{\psi},\overline{\varrho}}(n)
\end{align*}
together with the relation
$$c_{2-k}^+(tn,t,\psi,\varrho)=(tM)^{1-k}\frac{A_{2-k}(k-1;\varrho)}{k-1}
a_{2-k}(k-1;n,\overline{\psi},\varrho)n^{1-k}, \mbox{ for } n>0,$$
and $c_{2-k}^+(n,t,\psi,\varrho)=0$ if $t\nmid n$. This completes the proof of the theorem.
\end{proof}

\begin{proof}[Proof of Theorem \ref{teo-fourier-exp-mock-2,I,I}]
By \eqref{eq-Fou-exp-preimage-wt2-triv-char} we have
\begin{equation*}
\sum_{n=0}^{\infty}c_{0}^+(n,t,\bfo_L,\bfo_M)q^n
 =\frac{\pi \log(t)}{M} \prod_{p\mid L}(1-p^{-1}) \prod_{p\mid M}(1-p^{-1})  \\
+\frac{\pi}{M^2}\sum_{n=1}^{\infty}\sigma_1^{\bfo_L,\bfo_M}(n)n^{-1}(q^n-q^{tn}). \nonumber 
\end{equation*}
This already implies the desired formula for $c_{0}^+(0,t,\bfo_L,\bfo_M)$. For $n>0$, this also implies the relation
$$c_{0}^+(tn,t,\bfo_L,\bfo_M)=\frac{\pi}{M^2}\sigma_1^{\bfo_L,\bfo_M}(tn)(tn)^{-1}-\frac{\pi}{M^2}\sigma_1^{\bfo_L,\bfo_M}(n)n^{-1},$$
and
$$c_{0}^+(n,t,\bfo_L,\bfo_M)=\frac{\pi}{M^2}\sigma_1^{\bfo_L,\bfo_M}(n)n^{-1}, \mbox{ if }t\nmid n.$$
This gives the desired formula for $c_{0}^+(n,t,\bfo_L,\bfo_M)$ and completes the proof of the theorem.
\end{proof}

\begin{proof}[Proof of Theorem \ref{teo-four-exp-1}]
The formula for $c_{1}^+(0,t,\psi,\varrho)$ follows from \eqref{eq-four-exp-preimage-wt1}, together with 
\eqref{eq-derivative-C}, \eqref{eq-derivative-D} and \eqref{eq-derivative-D-trivial}. Further, from \eqref{eq-four-exp-preimage-wt1}, together with the fact that \eqref{eq-detivative-A} equals \eqref{eq-derivative-A-formula} and \eqref{eq-derivative-B} equals \eqref{eq-derivative-B-formula}, we get 
\begin{eqnarray*}
\sum_{n=1}^{\infty}c_{1}^+(n,t,\psi,\varrho)q^n &=&
-\frac{2\pi i }{M}W(\varrho^0) \sum_{n=1}^{\infty} \bigg( \sigma_0^{\overline{\psi},\overline{\varrho}}(n) \left(\log\left(\frac{\pi}{M}\right)-\Gamma'(1)-\log(tnM)\right) \\
 & &+2 \sum_{0<c\mid n}\overline{\psi}\left(\frac{n}{c}\right)\log(c)
\sum_{0<d\mid \gcd(\ell_{\varrho},c)}d\, \mu\left(\frac{\ell_{\varrho}}{d}\right)
\varrho^0\left(\frac{\ell_{\varrho}}{d}\right)\overline{\varrho^0}\left(\frac{c}{d}\right) \bigg)q^{tn}.\nonumber
\end{eqnarray*}
The desired formula for $c_{1}^+(n,t,\psi,\varrho)$, when $n>0$, follows from this together with \eqref{eq-Gamma-derivative-values}. This completes the proof of the theorem.
\end{proof}

\section{Examples}\label{sec-examples}

We shall now illustrate the results of Sections \ref{sect-const-preimage} and \ref{sect-four-expansions}  by giving some simple and nice examples.

\subsection{Level one Eisenstein series} For every even integer $k>2$ we have the normalized (i.e.~with constant Fourier coefficient equal to one) Eisenstein series of weight $k$
$$E_k(z):=1-\frac{2k}{B_k}\sum_{n=1}^{\infty}\sigma_{k-1}(n)q^n,$$
which is a modular form in $M_k(1,\bfo_1)$. Here $B_k$ denotes the $k$-th Bernoulli number and $\sigma_{k-1}(n):=\sigma_{k-1}^{\bfo_1,\bfo_1}(n)$.
Since $\zeta(k)=-\tfrac{B_k (2\pi i)^k }{2\, k!}$ (see, e.g., \cite{MIY}, Theorem 3.2.3, p.~90), one has the identity
$$E_k^{\bfo_1,\bfo_1,1}(z)=\zeta(k)E_k(z).$$
Hence, it follows from Theorem \ref{teo-fourier-exp-mock-geq2} that the function
$$\widetilde{E}_k(z):=\frac{1}{\zeta(k)}G^{\bfo_1,\bfo_1,1}_{2-k}(z)=-\frac{ 2^{2-k}\pi i^{k}}{\zeta(k)(k-1)}\bigg(\zeta(k-1)+\sum_{n=1}^{\infty}n^{1-k}\sigma_{k-1}(n)q^n\bigg)$$
is a mock modular form whose shadow equals $E_k(z)$. As mentioned in the introduction, these forms have also been constructed in \cite{BGKO}. Our formulas agree with theirs.

\subsection{Hecke's Eisenstein series associated to imaginary quadratic fields}\label{sect-Eis-wt-one}

Let $K=\QQ(\sqrt{D})$ be an imaginary quadratic field of discriminant $D<0$. Let us denote by $\mathcal{O}_D$ the ring of integers of $K$, by $u(D)$ the number of units in this ring, and by $h(D)$ its class number. For an integer $n>0$ we define $R_D(n)$ to be the number of integral ideals of norm $n$ in $\mathcal{O}_D$. One can then define Hecke's Eisenstein series as
$$E_{1,D}(z):=\frac{h(D)}{u(D)}+\sum_{n=1}^{\infty}R_D(n)q^n.$$
This is a holomorphic modular form in $M_1(|D|,\psi_D)$, where $\psi_D$ denotes the quadratic Dirichlet character associated to the field extension $\QQ(\sqrt{D})$ of $\mathbb{Q}$ (more concretely, one has $\psi_D(n)=\left(\frac{D}{n}\right)$, where $\left(\frac{\cdot}{\cdot }\right)$ denotes the Kronecker symbol).  By using the identity
\begin{equation}\label{eq-cont-ideals}
R_D(n)=\sum_{0<c\mid n}\psi_D(c),
\end{equation}
which is equivalent to the factorization $\zeta_K(s)=\zeta(s)L(\psi_D,s)$, where $\zeta_K(s)$ denotes the Dedekind zeta function associated to $K$ (see, e.g., \cite{COH}, Proposition 10.5.5, p.~219), and the class number formula
\begin{equation}\label{eq-class-number-formula}
L(0,\psi_D)=\frac{2h(D)}{u(D)}
\end{equation}
(see, e.g., \cite{COH}, part (3) of Theorem 10.5.1, p.~217), we get
$$E_1^{\psi_D,\bfo_1,1}(z)=-2\pi i E_{1,D}(z).$$
By Theorem \ref{teo-preimage} the function
$$\widetilde{E}_{1,D}(z):=(2\pi i)^{-1}G_1^{\psi_D,\bfo_1,1}(z)+(\log(\pi)+\gamma)E_{1,D}(z)$$
is a mock modular form whose shadow equals $E_{1,D}(z)$. Using Theorem \ref{teo-four-exp-1}, together with \eqref{eq-cont-ideals} and \eqref{eq-class-number-formula}, one obtains the Fourier expansion
$$\widetilde{E}_{1,D}(z)=\sum_{n=0}^{\infty}R^+_D(n)q^n,$$
where
$$R^+_D(0):=\left(2\log(2)+\log(\pi)+\gamma\right)\frac{h(D)}{u(D)}-L'(0,\psi_D),$$
and
$$R^+_D(n):=R_D(n)\log(n)-2\sum_{0<c\mid n}\psi_D(\tfrac{n}{c})\log(c), \mbox{ for }n>0.$$

The coefficients $R_D^+(n)$ have first been computed adelically  by Kudla, Rapoport and Yang  in the case where  $|D|>3$ is prime and $|D|\equiv 3$ mod 4 (see \cite{KRY99}, Theorem 1, p.~349), and by Schofer in the case where $|D|>3, |D|\equiv 3$ mod 4 (see \cite{SCH}, Theorem 4.1, pp.~30-31). The computation of $R_D^+(n)$ in the general case can be found in \cite{KY10} (Theorem 7.2 on p.~2305).  The mock modular form that one obtains from \cite{KRY99} equals $\widetilde{E}_{1,D}(z)-\log(|D|)E_{1,D}(z)$ in our notation. The fact that our  formulas agree with theirs follows from the following result. We define $$\Lambda(s,\psi_D):=\pi^{-(s+1)/2}\Gamma(\tfrac{s+1}{2})L(s,\psi_D)$$ and use the notation $\ord_p(n)$ for the $p$-adic order of an integer $n$.

\begin{proposition}
We have
$$R^+_D(0)=\frac{2h(D)}{u(D)}\left(\frac{\Lambda'(1,\psi_D)}{\Lambda(1,\psi_D)}+\log(|D|)\right),$$
and for $n>0$ we have the identity
$$R^+_D(n)=-R_D(n)\sum_{\substack{ p\mid n, \\ \psi_D(p)=0}}\log(p)\ord_p(n)-\sum_{\substack{p\mid n,\\ \psi_D(p)=-1}}\log(p)(\ord_p(n)+1)R_D(\tfrac{n}{p}),$$
where both sums run only over the prime divisors of $n$.
\end{proposition}
\begin{proof}
By using the functional equation of $\Lambda(s,\psi_D)$, namely
$$\Lambda(1-s,\psi_D)=\frac{i\, |D|^s}{W(\psi_D)}\Lambda(s,\psi_D)$$
(see, e.g., \cite{COH}, Theorem 10.2.14, p.~173), one gets
$$\frac{\Lambda'(1,\psi_D)}{\Lambda(1,\psi_D)}=-\frac{\Lambda'(0,\psi_D)}{\Lambda(0,\psi_D)}-\log(|D|)=\frac{1}{2}\left(\log(\pi)-\frac{\Gamma'(\frac{1}{2})}{\Gamma(\frac{1}{2})}-2\frac{L'(0,\psi_D)}{L(0,\psi_D)}\right)-\log(|D|).$$
From \eqref{eq-Gamma-derivative-values} and $\Gamma(\tfrac{1}{2})=\sqrt{\pi}$, we see that
$$-\frac{\Gamma'(\frac{1}{2})}{\Gamma(\frac{1}{2})}=2\log(2)+\gamma.$$
This, together with \eqref{eq-class-number-formula}, gives 
$$\frac{\Lambda'(1,\psi_D)}{\Lambda(1,\psi_D)}=\frac{u(D)}{2h(D)}R_D^+(0)-\log(|D|).$$
This implies the desired identity for $R_D^+(0)$. For the second identity, we start by writing
$$\sum_{0<c\mid n}\psi_D(\tfrac{n}{c})\log(c)=\log(n)R_D(n)-\sum_{0<c\mid n}\psi_D(c)\log(c),$$
which gives
$$R_D^+(n)=2\sum_{0<c\mid n}\psi_D(c)\log(c)-\log(n)R_D(n).$$
Writing $n=p_1^{a_1}\cdots p_r^{a_r}$ as a product of positive powers of different primes, and putting $n_j:=n p_j^{-a_j}$, we get
\begin{eqnarray*}
\sum_{0<c\mid n}\psi_D(c)\log(c)&=& \sum_{b_1=0}^{a_1}\cdots \sum_{b_r=0}^{a_r}\psi_D(p_1^{b_1}\cdots p_r^{b_r})\sum_{j=1}^rb_j\log(p_j)\\
&=& \sum_{j=1}^r\log(p_j) R_D(n_j)\sum_{b_j=0}^{a_j}b_j\psi_D(p_j^{b_j}),
\end{eqnarray*}
where we have used \eqref{eq-cont-ideals} for the last equality. A simple computation gives
$$\sum_{b_j=0}^{a_j}b_j\psi_D(p_j^{b_j})=\left\{
\begin{array}{ll}
\frac{a_j}{2}, & \mbox{ if }\psi_D(p_j)=-1 \mbox { and }a_j \mbox{ is even},\\
-\frac{(a_j+1)}{2},  & \mbox{ if }\psi_D(p_j)=-1 \mbox { and }a_j \mbox{ is odd},\\
0, & \mbox{ if }\psi_D(p_j)=0,\\
\frac{a_j(a_j+1)}{2}, & \mbox{ if }\psi_D(p_j)=1.
\end{array}
\right.$$
Similarly, we have
$$\log(n)R_D(n)=\sum_{j=1}^ra_j\log(p_j)R_D(n_j)R_D(p_j^{a_j}),$$
with
$$R_D(p_j^{a_j})=\left\{
\begin{array}{ll}
1, & \mbox{ if }\psi_D(p_j)=-1 \mbox { and }a_j \mbox{ is even},\\
0, & \mbox{ if }\psi_D(p_j)=-1 \mbox { and }a_j \mbox{ is odd},\\
1, & \mbox{ if }\psi_D(p_j)=0,\\
a_j+1, & \mbox{ if }\psi_D(p_j)=1.
\end{array}
\right.$$
Putting all this together, we obtain
$$R_D^+(n)=-\sum_{\substack{ j=1,\ldots,r \\ \psi_D(p_j)=0}}\log(p_j)a_jR_D(n_j)-\sum_{\substack{j=1,\ldots,r\\ \psi_D(p_j)=-1}}\log(p_j)(a_j+1)\delta(a_j)R_D(n_j),$$
where $\delta(a_j)=0$ if $a_j$ is even and $\delta(a_j)=(a_j+1)$ if $a_j$ is odd. If $\psi_D(p_j)=0$, one has the identity $R_D(n_j)=R_D(n)$. Furthermore, when $\psi_D(p_j)=-1$ one has the identity $\delta(a_j)R_D(n_j)=(a_j+1)R_D(\tfrac{n}{p})$. This implies the desired result and completes the proof of the proposition.
\end{proof}

\subsection{Certain theta series of weight one}

In some particular cases one can use the computations in the previous subsection in order to  construct mock modular forms whose shadows are theta series of weight one. More precisely, given an imaginary quadratic field $K=\QQ(\sqrt{D})$ of discriminant $D<0$, one can construct the theta series 
$$\vartheta_D(z)=\sum_{x\in \mathcal{O}_D}e^{2\pi i  (x\overline{x})z}.$$
If $K$ has class number one, which is the case if and only if $$D\in\{-3,-4,-7,-8,-11,-19,-43,-67,-163\},$$ then we have  $\vartheta_D(z)=u(D)E_{1,D}(z)$, and we can use the computations of the previous section to construct a pre-image of $\vartheta_D(z)$. \\
To give one explicit example, we choose $D=-4$, in which case we have
$$\vartheta_{-4}(z)=\Theta^2(z),$$
where 
$$\Theta(z):=\sum_{x\in \Z}q^{x^2}.$$
It follows from our computations, together with the identity
$$4R_{-4}(n)=r_2(n):=\{(x,y)\in \Z^2:x^2+y^2=n\},$$
that
\begin{eqnarray*}
\widetilde{\Theta^2}(z)&:=&2\log(2)+\log(\pi)+\gamma-4\, L'(0,\psi_{-4})\\
& & +\sum_{n=1}^{\infty}\left(r_2(n)\log(n)-8\sum_{0<c\mid n}\psi_{-4}(\tfrac{n}{c})\log(c)\right) q^n
\end{eqnarray*}
is a mock modular form whose shadow equals $\Theta^2(z)$. Note that here one can also use the identity
$$L'(0,\psi_{-4})=\log\left(\frac{\Gamma(\frac{1}{4})}{\Gamma(\frac{3}{4})}\right)-\log(2)$$
(see, e.g., \cite{COH}, Corollary 10.3.2 and Theorem 10.3.5, pp.~188--190).

\subsection{Pre-images of $\Theta^4(z)$, $\Theta^6(z)$ and $\Theta^8(z)$}

A result of Rankin in \cite{RAN} implies that, for $k\geq 1$ an integer, the form $\Theta^{2k}(z)$ is a linear combination of Eisenstein series if and only if $k\in\{ 1,2,3,4\}$. In particular, for $n\geq 1$ one has the identities
\begin{eqnarray*}
r_4(n)&=& 8\left(\sigma_1(n)-4\, \sigma_1(\tfrac{n}{4})\right),\\
r_6(n)&=& 4\left( 4\, \sigma_2^{\psi_{-4},\bfo_1}(n)-\sigma_2^{\bfo_1,\psi_{-4}}(n)\right)  , \\
r_8(n)&=& 16\left(\sigma_3(n)-2\, \sigma_3(\tfrac{n}{2})+16\, \sigma_3(\tfrac{n}{4})\right),
\end{eqnarray*}
where $r_{2k}(n)$ is the number of representations of $n$ as sum of $2k$ squares (see, e.g., \cite{RANbook}, identities (7.4.23), (7.4.24) and (7.4.25), pp.~241-242). 
The above corresponds in our notation to the identities
\begin{eqnarray*}
\Theta^4(z)&=& -\frac{2}{\pi^2}E_2^{\bfo_1,\bfo_1,4}(z),\\
\Theta^6(z)&=& -\frac{4i}{\pi^3}\left( E_3^{\psi_{-4},\bfo_1,1}(z)+8i E_3^{\bfo_1,\psi_{-4},1}(z)\right)  , \\
\Theta^8(z)&=& \frac{6}{\pi^4}\left( E_4^{\bfo_1,\bfo_1,1}(z)-2E_4^{\bfo_1,\bfo_1,2}(z)+16E_4^{\bfo_1,\bfo_1,4}(z)\right),
\end{eqnarray*}
where, for the second equality, we used that $L(3,\psi_{-4})=\dfrac{\pi^3}{2^5}$ (see, e.g., \cite{COH}, p.~189). By using Theorems \ref{teo-fourier-exp-mock-2,I,I} and \ref{teo-fourier-exp-mock-geq2} one obtains the mock modular forms
\begin{eqnarray*}
\widetilde{\Theta^4}(z)&:=&-\frac{1}{2^3\pi}\left(2^4\log(2)+\sum_{n=1}^{\infty}n^{-1}r_4(n)q^n\right)\\
\widetilde{\Theta^6}(z)&:=&-\frac{1}{2^4\pi^2}\left(2^4L(2,\psi_{-4})+\sum_{n=1}^{\infty}n^{-2}\left( r_6(n)+2^3 \, \sigma_{2}^{\bfo_1,\psi_{-4}}(n)\right)q^n\right), \\
\widetilde{\Theta^8}(z) &:=&-\frac{1}{2^4\pi^3}\left(2^3 \zeta(3)+\sum_{n=1}^{\infty}n^{-3}r_8(n)q^n\right) ,
\end{eqnarray*}
whose shadows equal $\Theta^4(z)$, $\Theta^6(z)$ and $\Theta^8(z)$, respectively.

\section*{Acknowledgements}

The authors thank \"Ozlem Imamo\={g}lu, Martin Raum and \'Arp\'ad T\'oth for their comments on a previous version of this paper. Herrero was partially funded by the  Royal Swedish Academy of Sciences.
Von Pippich acknowledges support from the LOEWE research unit 
Uniformized Structures in Arithmetic and Geometry.

\end{document}